\documentclass{article}
\usepackage[margin=35mm]{geometry}

\usepackage{amsthm,amsmath,amssymb}
\usepackage{mathtools}

\newtheorem{theorem}{Theorem}[section]
\newtheorem{proposition}[theorem]{Proposition}

\newtheorem{lemma}[theorem]{Lemma}

\theoremstyle{remark}
\newtheorem*{remark}{Remark}
\newtheorem*{notation}{Notation}

\newcommand{\tL}{\mathtt 1}                            % typewriter 1
\newcommand{\tO}{\mathtt 0}                            % typewriter 0

\DeclareMathOperator{\e}{\mathrm{e}}                   % exp(2\pi ix)
\DeclareMathOperator{\dens}{\,\mathrm dens}            % asymptotic density
\DeclareMathOperator{\LandauO}{\mathcal O}             % Big O-notation
\DeclareMathOperator{\realpart}{\mathrm {Re}}
\DeclareMathOperator{\imagpart}{\mathrm{Im}}
\DeclareMathOperator{\ud}{\mathrm d\!}              %upright d in an integral

\title
{A lower bound for Cusick's conjecture on the digits of $n+t$}
\author{Lukas Spiegelhofer\thanks{
The author acknowledges support by the project MuDeRa,
which is a joint project between the FWF (Austrian Science Fund) and the ANR (Agence Nationale de la Recherche, France).
Moreover, the author was supported by the FWF project F5502-N26, which is a part of the Special Research Program ``Quasi Monte Carlo methods: Theory and Applications''.}\\
Vienna University of Technology, Austria}
\date{}
\begin{document}
\maketitle
\begin{abstract}
Let $s$ be the sum-of-digits function in base $2$, which returns the number of $\mathtt 1$s in the base-2 expansion of a nonnegative integer.
For a nonnegative integer $t$, define the asymptotic density
\[
c_t=\lim_{N\rightarrow \infty}
\frac 1N\bigl\lvert\{0\leq n<N:s(n+t)\geq s(n)\}\bigr\rvert.\]
T.~W.~Cusick conjectured that $c_t>1/2$.
We have the elementary bound $0<c_t<1$; however, no bound of the form $0<\alpha\leq c_t$ or $c_t\leq \beta<1$, valid for all $t$, is known.
In this paper, we prove that $c_t>1/2-\varepsilon$ as soon as $t$ contains sufficiently many blocks of $\mathtt 1$s in its binary expansion.
In the proof, we provide estimates for the moments of an associated probability distribution; this extends the study initiated by Emme and Prikhod'ko (2017) and pursued by Emme and Hubert (2018).
\end{abstract}
\maketitle
\renewcommand{\thefootnote}{\fnsymbol{footnote}} 
\footnotetext{\emph{2010 Mathematics Subject Classification.} Primary: 11A63, 05A20; Secondary: 05A16,11T71}
\footnotetext{\emph{Key words and phrases.} Cusick conjecture, Hamming weight, sum of digits}
% 	11A63   	Radix representation; digital problems
%   05A20     Combinatorial inequalities

%   11T71     Algebraic coding theory; cryptography
%   05A16     Asymptotic enumeration

\renewcommand{\thefootnote}{\arabic{footnote}} 
\section{Introduction and main result}
It is an elementary problem of deceptive simplicity to study the behaviour of the base-$q$ digits of an integer under addition of a constant.
For example, it is clear that addition of the constant $1$ to an even integer in base $2$ replaces the rightmost digit $\tO$ by $\tL$, and addition of $1$ to an odd integer replaces the rightmost block of $\tL$s by a block of $\tO$s and the digit $\tO$ directly adjacent to this block by $\tL$.
Considerations of this kind can be carried out for each given constant $t$ in place of $1$, which gives a complete description of the digits of $n$ and $n+t$.

However, due to carry propagation the situation quickly turns into an unwieldy case distinction for growing $t$, and a general structural principle describing these cases is out of sight.
We therefore consider a simplification of this problem (which is still difficult) by studying a parameter associated to the binary expansion: the sum of digits $s(n)$ of $n$ in base $2$, which is just the number of $\tL$s occurring in the binary expansion of $n$.
More precisely, we are interested in the quantities
\[\delta(j,t)=\dens\,\{n\in\mathbb N:s(n+t)-s(n)=j\},\]
where $\dens A$ is the asymptotic density of a set $A\subseteq \mathbb N$, which exists in our case.
(In fact, the set in question is a finite union of arithmetic progressions, see B\'esineau~\cite{B1972} or the proof following~\eqref{eqn_density_recurrence}.)
Cusick's conjecture on the binary sum-of-digits function (private communication, 2011) concerns the values
\[c_t=\dens\,\{n\in\mathbb N:s(n+t)\geq s(n)\}=\delta(0,t)+\delta(1,t)+\cdots\]
and states that for all $t\geq 0$,
\begin{equation}\label{eqn_cusick}
c_t>1/2.
\end{equation}
This easy-to-state elementary problem appears to be difficult, despite its apparent simpleness.
Moreover, we think that it is not an artificial conjecture.
In our opinion, this combination of characteristics constitutes the beauty of this problem.
The proof below uses interesting techniques; this highlights the complex structure of the problem and further adds to the interest of this question. %xyz todo rewrite.

We note the partial results~\cite{DKS2016,EH2018,EH2018b,EP2017,S2019} on Cusick's conjecture, among which we find an almost-all result by Drmota, Kauers, and the author~\cite{DKS2016} and a central limit-type result by Emme and Hubert~\cite{EH2018}.

Cusick formulated his conjecture while he was working on the related \emph{Tu--Deng conjecture}~\cite{TD2011,TD2012}, which is relevant in cryptography: assume that $k$ is a positive integer and $t\in\{1,\ldots,2^k-2\}$. Then this conjecture states that
\[\Bigl \lvert\Bigl\{(a,b)\in\bigl\{0,\ldots,2^k-2\bigr\}^2:a+b\equiv t\bmod 2^k-1, s(a)+s(b)<k\Bigr\}\Bigr \rvert\leq 2^{k-1}.\]
Partial results are known, see~\cite{CLS2011,DY2012,F2012,FRCM10,SW2019,TD2011},
but the full conjecture is still open.
Besides an almost-all result on Tu and Deng's conjecture~\cite{SW2019}, Wallner and the author proved in that paper that this conjecture in fact implies Cusick's conjecture.

We return to Cusick's conjecture and begin with the case $t=1$.
From the introductory observation we obtain $s(n+1)-s(n)=1-\nu_2(n+1)$, where $\nu_2(m)=\max\{k\geq 0:2^k\mid m\}$,
which implies that $\delta(\cdot,1)$ describes a geometric distribution with mean $0$: we have
\[\delta(j,1)=\begin{cases}0,&j>1;\\2^{j-2},&j\leq 1,\end{cases}\]
and therefore $c_1=3/4$.
In other words, the sum of digits of $n+1$ is smaller than the sum of digits of $n$ if and only if $n\equiv 3\bmod 4$, since only in this case we lose at least one $\tL$ by replacing the rightmost block $\tO\tL^k$ by $\tL\tO^k$ in the binary expansion.

Next, we consider the general case $t\in\mathbb N$.
It follows from a recurrence due to B\'esineau~\cite{B1972} that the values $\delta(j,t)$ satisfy the following recurrence for all $k\in\mathbb Z$ and $t\geq 0$:
%{{{eqn_density_recurrence
\begin{equation}\label{eqn_density_recurrence}
\begin{aligned}
%\delta(k,1)&=\begin{cases}2^{k-2},&k\leq 1,\\0&\mbox{otherwise,}\end{cases}
%\\
\delta(j,2t)&=\delta(j,t),\\
\delta(j,2t+1)&=\frac 12 \delta(j-1,t)+\frac 12\delta(j+1,t+1).
\end{aligned}
\end{equation}
%}}}eqn:density_recurrence
The proof of the first identity is as follows:
we have the disjoint union 
\begin{align*}
\{n\in\mathbb N:s(n+2t)-s(n)=k\}
&=
2\{n\in\mathbb N:s(2n+2t)-s(2n)=k\}
\\&\cup \left(2\{n\in\mathbb N:s(2n+1+2t)-s(2n+1)=k\}+1\right),
\end{align*}
and using the identities $s(2n)=s(n)$ and $s(2n+1)=s(n)+1$, the first line of the recurrence follows. 
In an analogous way, the second line can be proved.
This proof also shows inductively that the sets defining $\delta(j,t)$ are finite unions of arithmetic progressions.

Using the recurrence~\eqref{eqn_density_recurrence}, we verified~\eqref{eqn_cusick} by numerical computation for all $t<2^{30}$,
yielding the minimal value $18169025645289/2^{45}=0.516394\ldots$ at the position\\ $t=(111101111011110111101111011111)_2$ and at the position $t^R$ obtained by reversing the base-$2$ expansion of $t$.
(By a result of Morgenbesser and the author~\cite{MS2012} we always have $\delta(j,t)=\delta(j,t^R)$.)

Using a generating function approach and Chebyshev's inequality, Drmota, Kauers, and the author~\cite{DKS2016} obtained an almost-all result for Cusick's conjecture:
for all $\varepsilon>0$, we have
\[\lvert\{t<T:1/2<c_t<1/2+\varepsilon\}\rvert=T-\LandauO(T/\log T).\]

Moreover, the probability distribution defined by $\mu_t:j\mapsto\delta(j,t)$ for given $t$ was studied by Emme and Hubert~\cite{EH2018,EH2018b}, continuing work by Emme and Prikhod'ko~\cite{EP2017}.
In~\cite{EH2018}, Emme and Hubert considered the moments of $\mu_t$ and proved a central limit law.
We introduce the notation
$a_X(\lambda)=\sum_{0\leq i\leq \lambda}X_i2^i$
for $X\in\{0,1\}^{\mathbb N}$ and $\lambda\geq 0$,
and we write $\Phi(x)=\frac 1{\sqrt{2\pi}}\int_{-\infty}^x e^{-x^2/2}\ud x$.
Then their result states the following.
For almost all $X$ with respect to the balanced Bernoulli measure,
we have
\begin{equation}\label{eqn_EH}
\lim_{\lambda\rightarrow \infty}\dens\left\{n\in\mathbb N:\frac{s(n+a_X(\lambda))-s(n)}{\sqrt{\lambda/2}}\leq x\right\}=\Phi(x)
\quad\mbox{for all }
x\in\mathbb R.
\end{equation}
%This answers a question left open by Drmota, Kauers and the author~
%\cite[Section 4.5]{DKS2016} in the affirmative.

In particular, excluding the negligible case that $s(n+a_X(k))-s(n)=0$ and considering $x=0$, this statement implies that 
\begin{equation}\label{eqn_EH_almost_surely}
\lim_{k\rightarrow\infty}c_{a_X(k)}=1/2
\end{equation}
almost surely. Note that this latter result does not follow directly from the Drmota--Kauers--Spiegelhofer result~\cite{DKS2016}, since our error term is not strong enough. On the other hand, the theorem by Emme and Hubert does not give us a statement of the form $c_t>1/2$ as in~\cite{DKS2016}.

From~\eqref{eqn_EH} we obtain the result that $c_t>1/2-\varepsilon$ for almost all $t$ with respect to asymptotic density.
The proof of this fact is by contradiction:
assume that $c_t\leq 1/2-\varepsilon$ for at least $2^{\lambda+1} \delta$ many $t<2^{\lambda+1}$ and infinitely many $\lambda$, where $\delta>0$.
Define
\[A_k=\{X\in\{0,1\}^{\mathbb N}:c_{a_X(\lambda)}>1/2-\varepsilon\textrm{ for all }\lambda\geq k\}.\]
By the almost sure convergence to $1/2$ and since the sequence of sets $A_k$ is ascending, we have
$\mu(A_N)>1-\delta$ for some $N$.
Then for all $X\in A_N$ and $\lambda>N$ we have
$c_{a_X(\lambda)}>1/2-\varepsilon$.
By definition of the balanced Bernoulli measure, 
there exist at least $(1-\delta)2^{\lambda+1}$ many $t<2^{\lambda+1}$ such that $c_t>1/2-\varepsilon$. This is a contradiction to our assumption,
by which there exists $\lambda>N$ such that $c_t\leq 1/2-\varepsilon$ for at least $2^{\lambda+1}\delta$ many $t<2^{\lambda+1}$.

While this result clearly also follows from the theorem by Drmota, Kauers, and the author, it is this particular formulation that we want to sharpen.
Our main theorem gives a lower bound for Cusick's conjecture for all $t$ not contained in a very small exceptional set having a simple structure.
In this theorem and in the following, we will be concerned with \emph{blocks} of $\tO$s or $\tL$s in the binary expansion of $t$; by this, we will always mean contiguous blocks of maximal size, where we omit the lowest block of $\tO$s for even integers $t$.
In particular, if we have the binary expansion $t=(\tL^{m_0}\tO^{n_0}\tL^{m_1}\tO^{n_1}\cdots \tL^{m_{\ell-1}}\tO^{n_{\ell-1}}\tL^{m_\ell}\tO^{n_\ell})_2$ with positive integers $m_i$ and $n_i$ (with the exception of $n_\ell$, which may be zero), then $t$ contains $\ell$ blocks of $\tO$s and $\ell+1$ blocks of $\tL$s.

Our main result is the following lower bound for $c_t$ for many values $t$.
\begin{theorem}\label{thm_main}
For all $\varepsilon>0$ there exists an $L\geq 0$ such that the following holds:
if the binary expansion of $t\in\mathbb N$ contains at least $L$ blocks of $\tL$s, then
\[c_t>1/2-\varepsilon.\]
In particular, for all $\varepsilon>0$ there exist $\delta>0$ and $C>0$ such that for $T\geq2$,
\[\lvert\{0\leq t<T:c_t\leq 1/2-\varepsilon\}\rvert\leq C\log^\delta T.\]
\end{theorem}

The ``in particular''-part results from counting the number of integers with less than $L$ blocks of $\tL$s in its binary expansion.
A rough upper bound is given as follows:
up to $2^\lambda$, there are not more than $\lambda^{2L-2}$ many such natural numbers, since the length of each block of $\tL$s as well as the position of the least significant $\tL$ in each block is bounded by $\lambda$.

The error term $\log^\delta T$ should be compared to Drmota--Kauers--Spiegelhofer's~\cite{DKS2016} much weaker error term $T/\log T$.
Certainly, the statement $c_t>1/2-\varepsilon$ in Theorem~\ref{thm_main} is weaker than the bound $c_t>1/2$ in~\cite{DKS2016}, but the constant $1/2$ is optimal:
for all $\varepsilon>0$, we have $c_t<1/2+\varepsilon$ for almost all $t$ with respect to asymptotic density (this follows, as above, from~\cite{EH2018}, or from~\cite[Theorem~1]{DKS2016}).

From Theorem~\ref{thm_main} we also obtain~\eqref{eqn_EH_almost_surely} almost surely, since the measure of the set of $X\in\{0,1\}^{\mathbb N}$ having only finitely many blocks of $\tL$s is zero.

Moreover, we note that Theorem~\ref{thm_main} significantly sharpens the main theorem in the recent paper~\cite{S2019} by the author:
in that paper, it was proved that $c_t+c_{t'}>1-\varepsilon$ if $t$ contains many blocks of $\tL$s; here $t'=3\cdot 2^\lambda-t$, where $2^\lambda\leq t<2^{\lambda+1}$.
The new Theorem~\ref{thm_main} gives a bound for \emph{individual} values $c_t$.

Finally, we note that the proof presented below allows to explicitly compute a bound $L=L(\varepsilon)$ for Theorem~\ref{thm_main}.
This is the case since all of the implied constants appearing in the proof are effective.

\begin{notation}
In this paper, $0\in\mathbb N$.
For an integer $n>0$, we use the notation $\nu_2(n)$ to denote the largest $k$ such that $2^k\mid n$.
We will use Big O notation, employing the symbol $\LandauO$.
For an integer $t\geq 1$, the \emph{number of blocks in} $t$ is the number of blocks of $\tL$s in the binary expansion of $t$ plus the number of blocks of $\tO$s in the proper binary expansion of $t/2^{\nu_2(t)}$. Clearly, if $\ell$ is the number of blocks of $\tL$s in the binary expansion of $t$, then $2\ell-1$ is the number of blocks in $t$.
We also define the number of blocks in $0$ to be $0$.
The variable $r$ is used to denote the number of blocks in $t$.
We let $\e(x)$ denote $e^{2\pi ix}$ for real $x$, and $\lVert x\rVert=\min_{k\in\mathbb Z}\lvert x-k\rvert$ is the distance to the nearest integer.
For convenience, we define the maximum over an empty index set to be $0$.
\end{notation}

The remainder of this paper is dedicated to the proof of Theorem~\ref{thm_main}.

\section{Proof of the main theorem}
The proof consists of several steps.
We consider the characteristic function $\gamma_t$ of a certain probability distribution. Using the link between $\gamma_t$ and $c_t$ expressed by~\eqref{eqn_ct_rep}, we see that we have to find upper bounds for $\imagpart \gamma_t$.
We do so in two stages: for $\vartheta$ not close to $\mathbb Z$,
we estimate the absolute value of $\gamma_t(\vartheta)$ using a matrix identity from~\cite{MS2012}.
For $\vartheta$ close to $\mathbb Z$, we estimate the imaginary part of $\gamma_t(\vartheta)$ using the link~\eqref{eqn_imagpart_moments} to the \emph{moments} $m_k(t)$.
The principal part of the proof is concerned with finding upper bounds for these moments (captured in Proposition~\ref{prp_moment_bounds}),
thus extending the study performed by Emme and Prikhod'ko~\cite{EP2017}, and Emme and Hubert~\cite{EH2018,EH2018b}.
%{{{ Subsection: Relating $c_t$ to a characteristic function
\subsection{Relating $c_t$ to a characteristic function}
For $\vartheta\in\mathbb R$ and $t\geq 0$, we define
\[\gamma_t(\vartheta)=\lim_{N\rightarrow\infty}\frac 1N\sum_{0\leq n<N}\e\bigl(\vartheta s(n+t)-\vartheta s(n)\bigr).\]
These limits exist, see B\'esineau~\cite{B1972},
and we have 
\begin{equation}\label{eqn_gamma_sum}
\gamma_t(\vartheta)=\sum_{-\infty\leq j<m}\delta(j,t)\e(j\vartheta)
\end{equation}
for some $m$ (which can be shown by induction easily).
We have
\begin{equation}\label{eqn:density_integral}
 \delta(j,t)
 =
 \int_0^1 \gamma_t(\vartheta) \e(-\vartheta j) \,\mathrm d\vartheta
\end{equation}
(see~\cite{MS2012}); using these identities and a geometric sum, we will prove the following identity.
\begin{proposition}\label{prp_ct_integral}
Let $t\geq 0$.
We have
%{{{equation:integral representation of c_t
\begin{equation}\label{eqn_ct_rep}
c_t = \frac 12 + \frac{\delta(0,t)}2 +
\frac 12
\int_0^1
\imagpart \gamma_t(\vartheta)\cot(\pi \vartheta) \ud \vartheta,
\end{equation}
%}}}
where the integrand is a bounded function.
\end{proposition}
In the proof of this statement,
we are also going to use the following fact.
\begin{lemma}
For $k\geq 1$ we have
\begin{equation}\label{eqn:integral_identity}
\int_0^1
\sin(2\pi k\vartheta)\cot(\pi\vartheta)\ud\vartheta
=1
,
\end{equation}
where the integrand is bounded on $(0,1)$.
\end{lemma}
\begin{proof}
For $k\geq 1$ and $\vartheta\in(0,1)$, we have
\begin{align*}
&\sin(2\pi k\vartheta)\cot(\pi\vartheta)
\\&=
\sin\bigl(2\pi(k-1)\vartheta\bigr)\cos(2\pi\vartheta)\cot(\pi\vartheta)
+
\cos\bigl(2\pi(k-1)\vartheta\bigr)\sin(2\pi\vartheta)\cot(\pi\vartheta)
\\&=
\sin\bigl(2\pi(k-1)\vartheta\bigr)\bigl(1-2\sin^2(\pi\vartheta)\bigr)\frac{\cos(\pi\vartheta)}{\sin(\pi\vartheta)}
+
2\cos\bigl(2\pi(k-1)\vartheta\bigr)\sin(\pi\vartheta)\cos(\pi\vartheta)\frac{\cos(\pi\vartheta)}{\sin(\pi\vartheta)}
\\&=
\sin\bigl(2\pi(k-1)\vartheta\bigr)\cot(\pi\vartheta)
-
\sin\bigl(2\pi(k-1)\vartheta\bigr)\sin(2\pi\vartheta)
+
\cos\bigl(2\pi(k-1)\vartheta\bigr)\bigl(\cos(2\pi\vartheta)+1\bigr)
\\[2mm]
&=
\sin\bigl(2\pi(k-1)\vartheta\bigr)\cot(\pi\vartheta)
+
\cos(2\pi k\vartheta)
+
\cos\bigl(2\pi(k-1)\vartheta\bigr).
\end{align*}

Assume first that $k=1$.
The first summand is identically zero on $(0,1)$,
the integral from $0$ to $1$ of the second summand equals zero,
and the third summand is identically $1$.
For $k\geq 2$ the first summand is bounded by the induction hypothesis and its contribution to the integral is $1$.
The other summands contribute nothing to the integral.
The statement is therefore proved.
\end{proof}
\begin{proof}[Proof of Proposition \ref{prp_ct_integral}]
Let $m$ be so large that~\eqref{eqn_gamma_sum} holds.
Necessarily we have $\delta(j,t)=0$ for $j\geq m$.
It follows from (\ref{eqn:density_integral}) that
\begin{equation}\label{eqn_ct_summation}
c_t
=\sum_{0\leq j<m}\delta(j,t)=\int_0^1
\gamma_t(\vartheta)
\sum_{0\leq j<m} \e(-j\vartheta)\ud\vartheta
=
\int_0^1
\realpart
\gamma_t(\vartheta)
\frac{1-\e(-m\vartheta)}{1-\e(-\vartheta)}
\ud\vartheta;
\end{equation}
we have the formulas
\begin{equation}\label{eqn_cot_introduction}
\realpart \frac 1{1-\e(-\vartheta)}=\frac 12\quad\mbox{and}\quad
\imagpart \frac 1{1-\e(-\vartheta)}=\frac 12\cot(-\pi\vartheta).
\end{equation}

Since $\delta(j,t)=\LandauO\left(2^{-\lvert j\rvert}\right)$ for $j\rightarrow\infty$ (where the implied constant depends on $t$), we have
$\imagpart \gamma_t(\vartheta)=\sum_{k\in\mathbb Z}\delta(j,t)\sin(2\pi k\vartheta)=\LandauO(\vartheta)$ for $\vartheta\rightarrow 0$;
also, $\cot(-\pi\vartheta)=\LandauO(1/\vartheta)$.
Equations~\eqref{eqn_ct_summation} and~\eqref{eqn_cot_introduction} imply
\begin{equation}\label{eqn_ct_expansion}
\begin{aligned}
c_t
&=
\int_0^1 
\realpart
\frac{\gamma_t(\vartheta)}{1-\e(-\vartheta)}
-
\realpart
\frac{\gamma_t(\vartheta)\e(-m\vartheta)}{1-\e(-\vartheta)}
\ud\vartheta
\\&=
\int_0^1 
\realpart
\frac{\gamma_t(\vartheta)}{1-\e(-\vartheta)}
-\frac 12\realpart \bigl(\gamma_t(\vartheta)\e(-m\vartheta)\bigr)
\\&\hspace{10em}+
\frac 12\imagpart \bigl(\gamma_t(\vartheta)\e(-m\vartheta)\bigr)\cot(-\pi\vartheta)
\ud\vartheta
,
\end{aligned}
\end{equation}
where all occurring summands are bounded functions.
Since $\sum_{j<m}\delta(j,t)=1$, it follows that
\[\gamma_t(\vartheta)\e(-m\vartheta)
=\sum_{\ell\geq 1}a_\ell \e(-\ell\vartheta)
\]
for some nonnegative $a_\ell$ such that $\sum_{\ell\geq 1}a_\ell=1$.
Since $m$ is large enough, the integral over the second summand in the second line of~\eqref{eqn_ct_expansion} is zero.
We obtain
\[    c_t=\int_0^1 \realpart\frac{\gamma_t(\vartheta)}{1-\e(-\vartheta)}
+\frac 12\sum_{\ell\geq 1}
a_\ell\sin(-2\pi\ell\vartheta)\cot(-\pi\vartheta)\ud\vartheta.    \]
The partial sums $\sum_{1\leq \ell<L}a_\ell\sin(-2\pi\ell\vartheta)\cot(-\pi\vartheta)\ud\vartheta$ are bounded, uniformly in $L$, by an integrable function on $[0,1]$, therefore the statement follows by interchanging the summation and the integral, an application of the identity (\ref{eqn:integral_identity}), and another application of~\eqref{eqn_cot_introduction}.
\end{proof}
%}}}
\subsection{Moments of $\mu_t$}
Define
\[\widetilde m_k(t)=\sum_{j\in\mathbb Z}\delta(j,t)j^k.\]
The moment generating function is
\[M_t(x)=\sum_{k\geq 0}\frac{\widetilde m_k(t)}{k!}x^k.\]
The moments exist and the moment generating function is convergent for $t=1$ and $\lvert x\rvert<\log 2$;
this is just a geometric distribution and we obtain
\begin{equation}\label{eqn_M1}
M_1(x)=\frac{e^x}{2-e^{-x}}=1+x^2-x^3+\frac{19}{12}x^4-\cdots
\end{equation}
(see entry \texttt{A052841} in Sloane's OEIS\footnote{\texttt{http://oeis.org}}\ ).
By basic analytic combinatorics~\cite{FS2009}, we obtain
\begin{equation}\label{eqn_moment_1_asymp}
m_k(t)\sim c(\log 2)^{-k}
\end{equation} for some absolute $c$, as $k\rightarrow\infty$.

For $t\geq 2$, we note that the recurrence relation~\eqref{eqn_density_recurrence} implies that $\delta(j,t)=c2^j$ for $j<-\lambda$, where $2^\lambda\leq t<2^{\lambda+1}$. 
%\todo{check this and work out the details.}
This implies that $\widetilde m_k(t)$ exists for all $k$ and $\lvert \widetilde m_k(t)\rvert \ll \lambda^k + \widetilde m_k(1)$.
Considering also the series for the exponential function and the asymptotic estimate~\eqref{eqn_moment_1_asymp}, we see that the series for $M_t(x)$ is convergent as long as $\lvert x\rvert<\log 2$.

From~\eqref{eqn_density_recurrence} we wish to derive a recurrence for the moment generating functions.
We define
\[m_k(t)=\frac{\widetilde m_k(t)}{k!},\] such that $M_k(x)=\sum_{k\geq 0}m_k(t)x^k$.
\begin{lemma}
Assume that $\lvert x\rvert<\log 2$ and $t\geq 0$.
We have
\begin{equation}\label{eqn_MGF_recurrence}
\begin{aligned}
M_{2t}(x)&=M_t(x)\textrm{ and}\\
M_{2t+1}(x)&=\frac{e^x}2 M_t(x)+\frac{e^{-x}}2 M_{t+1}(x).
\end{aligned}
\end{equation}
In particular,
\begin{equation}\label{eqn_moments_recurrence}
\begin{aligned}
m_k(2t)&=m_k(t) \textrm{ and}\\
m_k(2t+1)&=\frac 12\sum_{0\leq \ell\leq k} \frac 1{\ell!}\bigl(m_{k-\ell}(t)+(-1)^\ell m_{k-\ell}(t+1)\bigr).
\end{aligned}
\end{equation}
\end{lemma}
Note that setting $t=0$, we obtain~\eqref{eqn_M1}.
\begin{proof}
The first line of~\eqref{eqn_MGF_recurrence} is trivial since $\delta(j,2t)=\delta(j,t)$.
Concerning the second line, we have
\[M_t(x)=\sum_{j\in\mathbb Z}\delta(j,t)e^{jx},\]
therefore by~\eqref{eqn_density_recurrence}
\begin{align*}
M_{2t+1}(x)&=\sum_{j\in\mathbb Z}\delta(j,2t+1)e^{jx}
=\frac 12 \sum_{j\in\mathbb Z}\delta(j-1,t)e^{jx}
+\frac12\sum_{j\in\mathbb Z}\delta(j+1,t+1)e^{jx}
\\&=\frac{e^{x}}2M_t(x)+\frac{e^{-x}}2M_{t+1}(x)
\end{align*}
after a shift of indices.
The ``in particular''-part follows from expanding Cauchy products.
\end{proof}

The first few moments are as follows: $m_0(t)=1$, $m_1(t)=0$, %(by induction using~\eqref{eqn_density_recurrence} or~\eqref{eqn_moments_recurrence})
and $m_2(t)$ satisfies the recurrence
\begin{align*}
m_2(0)=0,\quad
m_2(1)=1,\quad
m_2(2t)=m_2(t),\quad
m_2(2t+1)=\frac{m_2(t)+m_2(t+1)+1}2.
\end{align*}
This particular sequence also arises in a different context: it is the \emph{star-discrepancy} of the Van der Corput sequence in base $2$~\cite{DLP2005,S2018}. %Drmota, Larcher, Pillichshammer; Spiegelhofer.
It is known that
\begin{equation}\label{eqn_BF}
m_2(t)\leq \frac{\log t}{3\log 2}+1
\end{equation}
for all $t\geq 1$ (see Bejian and Faure~\cite{BF1978}).
The following interesting exact representation of $m_2(t)$ follows from Pro\u{\i}nov and Atanassov~\cite{PA1988}, and Beck~\cite{B2014} (as was pointed out to the author by the anonymous referee of the article~\cite{S2018}); see the remark after~\cite[Corollary~2.5]{S2018}:
if $t=\sum_{0\leq i\leq \nu}\varepsilon_i2^i$ with $\varepsilon_i\in\{0,1\}$, we have
\begin{equation}\label{eqn_second_moment_exact}
m_2(t)=\sum_{0\leq i\leq\nu}\varepsilon_i-\sum_{0\leq i<j\leq \nu}\varepsilon_i\varepsilon_j2^{i-j}.
\end{equation}
At this point we wish to emphasize the usefulness of the moments as opposed to $\delta(j,t)$.
In order to compute $\delta(j,t)$, we need to consider values $\delta(j+\ell,t')$ for large $\ell$ (depending on the number of $\tL$s and $\tO$s in the binary expansion of $t/2^{\nu_2(t)}$); for computing $m_k(t)$ we only need to consider
moments $m_i(t')$ for $i\leq k$.
In particular, $m_k$ is a $2$-\emph{regular} sequence~\cite{AS1992}, while this is not so clear and perhaps wrong for $\delta(j,\cdot)$.

Using Chebyshev's inequality and the bound~\eqref{eqn_BF} for $m_2(t)$
we can already find a nontrivial bound related to Cusick's conjecture:
with $\sigma=\sqrt{(\log t)/(3\log 2)+1}$ we have
$\sum_{\lvert j\rvert\leq K\sigma}\delta(j,t)\geq 1-1/K^2$.
In particular, choosing $K$ close to $\sqrt{2}$, we obtain
\begin{equation}\label{eqn_cusick_shifted}
\sum_{j\geq -\sqrt{\log t}-1}\delta(j,t)\geq 
\sum_{j\geq -\sqrt{\log t+2}}\delta(j,t)>1/2
\end{equation}
for $t\geq 1$.

We are going to establish recurrences for the values
\begin{align*}
a_k(t)&=m_k(t)+m_k(t+1),\\
b_k(t)&=m_k(t)-m_k(t+1),
\end{align*}
%Obtaining $m_k(t)$ is a tiny step at the very end of this process.
and the corresponding generating functions
\[F_t(x)=\sum_{k\geq 0}a_k(t)x^k\quad\mbox{and}\quad
G_t(x)=\sum_{k\geq 0}b_k(t)x^k.\]
By~\eqref{eqn_moments_recurrence}, we have
\begin{align*}
2a_k(2t)&=2m_k(t)+m_k(t)+m_k(t+1)
+\sum_{\substack{2\leq \ell\leq k\\2\mid \ell}}\frac 1{\ell!}a_{k-\ell}(t)
+\sum_{\substack{1\leq \ell\leq k\\2\nmid\ell}}\frac 1{\ell!}b_{k-\ell}(t)
\\
&=a_k(t)
+\sum_{\substack{0\leq \ell\leq k\\2\mid \ell}}\frac 1{\ell!}a_{k-\ell}(t)
+b_k(t)+\sum_{\substack{1\leq \ell\leq k\\2\nmid\ell}}\frac 1{\ell!}b_{k-\ell}(t)
\end{align*}
and
\begin{align*}
2b_k(2t)&=2m_k(t)-m_k(t)-m_k(t+1)
-\sum_{\substack{2\leq \ell\leq k\\2\mid \ell}}\frac 1{\ell!}a_{k-\ell}(t)
-\sum_{\substack{1\leq \ell\leq k\\2\nmid\ell}}\frac 1{\ell!}b_{k-\ell}(t)
\\
&=a_k(t)
-\sum_{\substack{0\leq \ell\leq k\\2\mid \ell}}\frac 1{\ell!}a_{k-\ell}(t)
+b_k(t)-\sum_{\substack{1\leq \ell\leq k\\2\nmid\ell}}\frac 1{\ell!}b_{k-\ell}(t).
\end{align*}
We want to write this as a matrix recurrence;
we define
\begin{align*}
C(x)&=\cosh(x)=\frac 12\left(e^x+e^{-x}\right)
=\sum_{j\geq 0}\frac{x^{2j}}{(2j)!};\\
S(x)&=\sinh(x)=\frac 12\left(e^x-e^{-x}\right)=\sum_{j\geq 0}\frac{x^{2j+1}}{(2j+1)!}.
\end{align*}
We are concerned with the matrix
\[M_0=\frac 12(A_0+B_0),\]
where
\[  A_0=\left(\begin{matrix}C(x)&S(x)\\-C(x)&-S(x)\end{matrix}\right)
\quad\mbox{and}\quad B_0=\left(\begin{matrix}1&1\\1&1\end{matrix}\right).  \]

We also study appending $\tL$ to the binary expansion.
\begin{align*}
2a_k(2t+1)&=m_k(t)+m_k(t+1)+2m_k(t+1)
+\sum_{\substack{2\leq \ell\leq k\\2\mid \ell}}\frac 1{\ell!}a_{k-\ell}(t)
+\sum_{\substack{1\leq \ell\leq k\\2\nmid\ell}}\frac 1{\ell!}b_{k-\ell}(t)
\\
&=a_k(t)
+\sum_{\substack{0\leq \ell\leq k\\2\mid \ell}}\frac 1{\ell!}a_{k-\ell}(t)
-b_k(t)+\sum_{\substack{1\leq \ell\leq k\\2\nmid\ell}}\frac 1{\ell!}b_{k-\ell}(t)
\end{align*}
and
\begin{align*}
2b_k(2t+1)&=m_k(t)+m_k(t+1)-2m_k(t+1)
+\sum_{\substack{2\leq \ell\leq k\\2\mid \ell}}\frac 1{\ell!}a_{k-\ell}(t)
+\sum_{\substack{1\leq \ell\leq k\\2\nmid\ell}}\frac 1{\ell!}b_{k-\ell}(t)
\\
&=-a_k(t)+
\sum_{\substack{0\leq \ell\leq k\\2\mid \ell}}\frac 1{\ell!}a_{k-\ell}(t)
+b_k(t)+\sum_{\substack{1\leq \ell\leq k\\2\nmid\ell}}\frac 1{\ell!}b_{k-\ell}(t).
\end{align*}
Clearly, we are interested in the matrix
\[M_1=\frac 12(A_1+B_1),\]
where
\[  A_1=\left(\begin{matrix}C(x)&S(x)\\C(x)&S(x)\end{matrix}\right)
\quad\mbox{and}\quad
B_1=\left(\begin{matrix}1&-1\\-1&1\end{matrix}\right).  \]
Using Cauchy products and the recurrence formula~\eqref{eqn_moments_recurrence},
we see that
\begin{equation}\label{eqn_momentcoeff_rec}
\begin{aligned}
\left(\begin{matrix}F_{2t}(x)\\G_{2t}(x)\end{matrix}\right)
&=M_0
\left(\begin{matrix}F_t(x)\\G_t(x)\end{matrix}\right)
\quad\mbox{and}\quad
\left(\begin{matrix}F_{2t+1}(x)\\G_{2t+1}(x)\end{matrix}\right)
=M_1
\left(\begin{matrix}F_t(x)\\G_t(x)\end{matrix}\right).
\end{aligned}
\end{equation}
Note that these identities are also valid for $t=0$
(since~\eqref{eqn_moments_recurrence} is also valid for $t=0$).

%{{{ subsection: Estimating the characteristic function using moments
\subsection{Estimating the characteristic function using moments}
%We have the following representation for $\lvert \vartheta\rvert<(\log 2)/(2\pi)$:
%\begin{equation}\label{eqn_gamma_moments}
%\imagpart\gamma_t(\vartheta)=\sum_{k\geq 0}m_{2k+1}(t)(-1)^k(2\pi\vartheta)^{2k+1}.
%\end{equation}
For $\vartheta\leq \vartheta_0$, where $\vartheta_0$ is defined later, we will use a representation of $\imagpart \gamma_t(\vartheta)$ in terms of moments.
For this, we use Taylor approximation of the sine function and~\eqref{eqn_gamma_sum}:
for all $j\in\mathbb Z$ there exists $\xi_j$ between $0$ and $2\pi \vartheta$ such that %(which will be around $r^{-1/2}$) such that

\begin{align*}
\imagpart\gamma_t(\vartheta)&=\imagpart\sum_{j\in\mathbb Z}\delta(j,t)\e(j\vartheta)=\sum_{j\in\mathbb Z}\delta(j,t)\sin(2\pi j\vartheta)\\
&=\sum_{j\in\mathbb Z}\delta(j,t)\sum_{0\leq k< K}
\frac{(-1)^k}{(2k+1)!}(2\pi\vartheta)^{2k+1}j^{2k+1}
+\sum_{j\in \mathbb Z}\frac {\delta(j,t)(-1)^K}{(2K+1)!}j^{2K+1}\xi_j^{2K+1}
\end{align*}
and therefore
\[
\left\lvert\imagpart\gamma_t(\vartheta)\right\rvert
\leq
\sum_{0\leq k<K}(2\pi \vartheta)^{2k+1}\left\lvert m_{2k+1}(t)\right\rvert
+\frac{(2\pi\vartheta)^{2K+1}}{(2K+1)!}
\sum_{j\in \mathbb Z}\delta(j,t)\lvert j\rvert^{2K+1}.
\]
Applying the Cauchy--Schwarz inequality to the sum
\[
\sum_{j\in\mathbb Z}\delta(j,t)\lvert j\rvert^{2K+1}
=
\sum_{j\in\mathbb Z}\sqrt{\delta(j,t)}\lvert j\rvert^K
\sqrt{\delta(j,t)} \lvert j\rvert^{K+1},
%\leq
%\left(
%\sum_{j\in\mathbb Z}\delta(j,t)j^2
%\right)^{1/2}
%\left(
%\sum_{j\in\mathbb Z}\delta(j,t)j^{4K}
%\right)^{1/2}
\]
we obtain
\begin{equation}\label{eqn_imagpart_moments}
\begin{aligned}
\bigl\lvert \imagpart \gamma_t(\vartheta)\bigr\rvert
&\leq
\sum_{0\leq k<K}(2\pi \vartheta)^{2k+1}\left\lvert m_{2k+1}(t)\right\rvert
\\&+(2\pi\vartheta)^{2K+1}\frac{\sqrt{(2K)!(2K+2)!}}{(2K+1)!}m_{2K}(t)^{1/2}m_{2K+2}(t)^{1/2}.
\end{aligned}
\end{equation}
The moments $m_{2K}$ and $m_{2K+2}$ will give us a factor $(K!(K+1)!)^{1/2}\geq K!$ in the denominator, which we will see later; this gain will enable us to prove that \emph{for all} $\varepsilon>0$, we have $c_t>1/2-\varepsilon$ for most $t$ (the exceptional set depending on $\varepsilon$).
%and the term $m_2(t)^{1/2}$ will contribute $r^{1/2}$.
%The integral from $0$ to $\vartheta_0$ has length $r^{-1/2}\varphi(k)$,
%therefore we can exploit the gain coming from the factor $\sqrt{(4K)!/(2K)!}/(2K+1)!$.

%Omitting the sign $(-1)^k$ and taking the absolute value of $\vartheta$, we obtain the series for $e^{2\pi\lvert\vartheta\rvert}$;
%since $\delta(j,t)\sim c2^j$ for $j\rightarrow -\infty$, we see that the family $(a_{j,k})_{\substack{j\in \mathbb Z\\k\geq 0}}$ defined by $a_{j,k}=\delta(j,t)\frac{(-1)^k}{(2k+1)!}(2\pi j \vartheta)^{2k+1}$ is summable as long as $e^{2\pi\lvert\vartheta\rvert}<2$.
%It follows that for these $\vartheta$ we may interchange the summations over $j$ and $k$, which clearly yields~\eqref{eqn_gamma_moments}.

\subsection{Upper bounds for the moments of $\mu_t$}

We wish to study repeated application of the recurrence~\eqref{eqn_moments_recurrence}, corresponding to appending a block of $\tO$s or $\tL$s to the binary expansion of $t$.

%We are going to capture the process of appending a block of digits with the help of generating functions.
%First, we consider appending the digit $\tL$.
%In this case, we consider the transformation of $a_k(t)$ and $b_k(t)$.
%Using~\eqref{eqn_moments_recurrence}, we obtain

Using an elementary proof by induction, we obtain
\[ A_0^m=\bigl(C(x)-S(x)\bigr)^{m-1}A_0=e^{-(m-1)x}A_0 \]
and
\[ B_0^m=2^{m-1}B_0 \]
for $m\geq 1$.
Moreover, $B_0A_0=\left(\begin{smallmatrix}0&0\\0&0\end{smallmatrix}\right)$
and $A_0B_0=\bigl(C(x)+S(x)\bigr)D_0=e^{x}D_0,$
where 
$D_0=\left(\begin{smallmatrix}1&1\\-1&-1\end{smallmatrix}\right)$.
Appending a block of $\tO$s of length $m$ corresponds to the matrix power $M_0^m$.

Noting that $A_0$ and $B_0$ do not commute, we consider all ordered products of $A_0$ and $B_0$ of length $m_0$;
since $B_0A_0$ vanishes, we are only interested in the products $A_0^{m_0-\ell} B_0^\ell$.
This yields
\begin{align*}
2^{m_0}M_0^{m_0}&=
\sum_{0\leq \ell\leq m_0}
A_0^{m_0-\ell}B_0^\ell 
=
e^{-(m_0-1)x}A_0+2^{m_0-1}B_0+
\sum_{1\leq \ell\leq m_0-1}
A_0^{m_1-\ell}B_0^\ell.
\end{align*}
We have
\begin{align*}
\sum_{1\leq \ell\leq m_1-1}
A_0^{m_0-\ell}B_0^\ell
&=
\sum_{1\leq \ell\leq m_0-1}
\e^{-(m_0-\ell-1)x}A_0
2^{\ell-1}B_0
\\&=
e^{(3-m_0)x}D_0
\sum_{0\leq \ell\leq m_1-2}
\left(2e^x\right)^\ell
=
e^x\frac{2^{m_0-1}-e^{(1-m_0)x}}{2-e^{-x}}D_0,
\end{align*}
which is also valid for $m=1$,
and therefore
\begin{align}\label{eqn_M0_power}
2M_0^{m_0}=
\left(\frac{e^{-x}}2\right)^{m_0-1}A_0
+B_0
+\frac {e^{x}}{2-e^{-x}}
\biggl(1-
\left(\frac{e^{-x}}2\right)^{m_0-1}
\biggr)D_0.
\end{align}
Also, we study appending a block of $\tL$s.
By induction, we obtain
\[ A_1^m=\bigl(C(x)+S(x)\bigr)^{m-1}A_1=e^{(m-1)x}A_1 \]
and
\[ B_1^m=2^{m-1}B_1 \]
for $m\geq 1$.
Moreover, $B_1A_1=\left(\begin{smallmatrix}0&0\\0&0\end{smallmatrix}\right)$
and $A_1B_1=\bigl(C(x)-S(x)\bigr)D_1=e^{-x}D_1,$
where 
$D_1=\left(\begin{smallmatrix}1&-1\\1&-1\end{smallmatrix}\right)$.
Appending a block of $\tL$s of length $m$ corresponds to the matrix power $M_1^m$.
Again, the matrices $A_1$ and $B_1$ do not commute, but $B_1A_1$ vanishes. This yields
\begin{align*}
2^{m_1}M_1^{m_1}&=
\sum_{0\leq \ell\leq m_1}
A_1^{m_1-\ell}B_1^\ell
=
e^{(m_1-1)x}A_1+2^{m_1-1}B_1+
\sum_{1\leq \ell\leq m_1-1}
A_1^{m_1-\ell}B_1^\ell.
\end{align*}
We have
\begin{align*}
\sum_{1\leq \ell\leq m_1-1}
A_1^{m-\ell}B_1^\ell 
&=
\sum_{1\leq \ell\leq m_1-1}
e^{(m_1-\ell-1)x}A_1
2^{\ell-1}B_1
\\&=
e^{(m_1-3)x}D_1
\sum_{0\leq \ell\leq m_1-2}
\left(\frac 2{e^x}\right)^\ell
=
\frac 1{e^x}\frac{2^{m_1-1}-e^{(m_1-1)x}}{2-e^x}D_1,
\end{align*}
which is also valid for $m=1$,
and therefore
\begin{align}\label{eqn_M1_power}
2M_1^{m_1}=
\left(\frac{e^x}2\right)^{m_1-1}A_1
+B_1
+\frac {e^{-x}}{2-e^x}
\biggl(1-
\left(\frac{e^x}2\right)^{m_1-1}
\biggr)D_1.
\end{align}

We are interested in the entries of these matrix powers;
they are generating functions in the variable $x$, convergent in the whole of $\mathbb C$, and we consider their coefficients.
First, we prove a statement on the low powers of $x$.

\begin{lemma}\label{lem_lower_coeffs}
Assume that $m\geq 1$.
Let
\[M_0^m=\left(\begin{matrix}\mathfrak a_0(x)&\mathfrak b_0(x)\\\mathfrak c_0(x)&\mathfrak d_0(x)\end{matrix}\right).\]
Then as $x\rightarrow 0$,
\begin{equation}\label{eqn_M0_coeffs}
\begin{aligned}
\mathfrak a_0(x)&=1+\frac{2^m-1}{2^{m+1}}x^2+\LandauO(x^3);&
\mathfrak b_0(x)&=\frac {2^m-1}{2^m}+\LandauO(x);\\
\mathfrak c_0(x)&=-\frac{2^m-1}{2^{m+1}}x^2+\LandauO(x^3);&
\mathfrak d_0(x)&=\frac 1{2^m}+\LandauO(x).
\end{aligned}
\end{equation}

Let
\[M_1^m=\left(\begin{matrix}\mathfrak a_1(x)&\mathfrak b_1(x)\\\mathfrak c_1(x)&\mathfrak d_1(x)\end{matrix}\right).\]
Then
\begin{equation}\label{eqn_M1_coeffs}
\begin{aligned}
\mathfrak a_1(x)&=1+\frac{2^m-1}{2^{m+1}}x^2+\LandauO(x^3);&
\mathfrak b_1(x)&=-\frac {2^m-1}{2^m}+\LandauO(x);\\
\mathfrak c_1(x)&=\frac{2^m-1}{2^{m+1}}x^2+\LandauO(x^3);&
\mathfrak d_1(x)&=\frac 1{2^m}+\LandauO(x).
\end{aligned}
\end{equation}
\end{lemma}
\begin{proof}
The proof of this statement is easy, using~\eqref{eqn_M0_power} and~\eqref{eqn_M1_power}, and the expansions of $(e^{\pm x})^m$, $e^{\pm x}/(2-e^{\mp x})$, $C(x)$, and $S(x)$.
We leave the details of this straightforward calculation to the reader.
\end{proof}

We also prove bounds for the error terms occurring in Lemma~\ref{lem_lower_coeffs}. 
That is, we need upper bounds for the coefficients
$\left[x^s\right]\mathfrak a_i(x)$ and 
$\left[x^s\right]\mathfrak c_i(x)$ for $s\geq 3$,
and for the coefficients
$\left[x^s\right]\mathfrak b_i(x)$ and 
$\left[x^s\right]\mathfrak d_i(x)$ for $s\geq 1$.

\begin{lemma}\label{lem_higher_coeffs}
Let $\mathfrak f\in\{\mathfrak a_0,\mathfrak b_0,\mathfrak c_0,\mathfrak d_0,\mathfrak a_1,\mathfrak b_1,\mathfrak c_1,\mathfrak d_1\}$.
Then for $k\geq 1$,
\[\left[x^k\right]\mathfrak f(x)\leq \frac 2{(\log 2)^k}.\]

\end{lemma}
\begin{proof}
%We treat the two cases 
%restrict ourselves to the case of appending zeros, the other case being analogous.
We first note that
%study the coefficients of
%\[\left(e^x/2\right)^m\]
%first.
%Clearly, we have 
\[\left[x^k\right]\left(e^x/2\right)^m
=\frac{m^k}{2^mk!}.\]
This function in $m$ attains its maximum at $m=k/\log 2$, yielding
\[\left[x^k\right]\left(e^x/2\right)^m\leq \left(\frac ke\right)^k\frac 1{k!}\frac 1{(\log 2)^k}\leq \frac 1{\sqrt{2\pi k}}\frac 1{(\log 2)^k},\]
using Robbins~\cite{R1955}.
Also, the third summands in~\eqref{eqn_M0_power} and~\eqref{eqn_M1_power} can be estimated by resorting to an asymptotic formula for the Fubini numbers (the coefficients of the exponential generating function for $1/(2-e^x)$).
Such an estimate follows easily from basic analytic combinatorics~\cite{FS2009}:
as $k\rightarrow\infty$, we have
\[\left[x^k\right] \frac 1{2-\e^x}=\frac{1}{2(\log 2)^{k+1}}(1+o(1)).\]
%We also obtain for all $k\geq 0$
%\[\left[x^k\right] \frac 1{2-\e^x}\leq \frac{1}{(\log 2)^k}.\]
Taking all singularities of $1/(2-e^x)$ into account, we obtain the exact formula~\cite{B1998}
\begin{equation}\label{eqn_Bailey}
\left[x^k\right] \frac 1{2-e^x}
=
\frac 12\sum_{n=-\infty}^\infty
\bigl(\log(2)+2\pi i n\bigr)^{-k-1},
\end{equation}
valid for $k\geq 1$.
The contribution of the terms $\lvert k\rvert \geq 1$ can be estimated using an integral: for $n\geq 1$ we have
\[\sum_{k\geq 1}k^{-n-1}\leq 1+\int_1^\infty x^{-n-1}\ud x=1+\frac 1n\leq 2.\]
Therefore 
\begin{equation}\label{eqn_fubini_explicit}
\left[x^k\right] \frac 1{2-\e^x}\leq
\frac 1{2(\log 2)^{k+1}}
+\frac 2{(2\pi)^{k+1}}
\leq \frac 1{(\log2)^k}
\end{equation}
for $k\geq 1$, the outer estimate also being valid for $k=0$.

For the estimation of coefficients of $M_i^m$,
we are also interested in partial sums.
We show the statements for $M_1$.
The case $M_0$ can be obtained replacing $\pm x$ by $\mp x$,
noting that for a generating function $H(x)$, the generating functions $H(x)$ and $H(-x)$ differ only by the sign of their respective coefficients.
We first show that $k\mapsto \left[x^k\right]1/(2-e^x)$ is nondecreasing, using~\eqref{eqn_Bailey}.
Passing from $k$ to $k+1$, the summand in~\eqref{eqn_Bailey} corresponding to $n=0$ increases by a quantity bounded below by $0.6$.
The other terms, in total, change by less, as we show now.
For an integer $n\geq 1$, we have
\begin{align*}
\left\lvert
\bigl(\log 2+2\pi i n\bigr)^{-k-2}-
\bigl(\log 2+2\pi i n\bigr)^{-k-1}
\right\rvert
&\leq 
\bigl\lvert\bigl(\log 2+2\pi i n\bigr)^{-k-1}\bigr\rvert
\bigl\lvert
\bigl(\log 2+2\pi i n\bigr)^{-1}-1
\bigr\rvert
\\&\leq
2\cdot (6n)^{-k-1},
\end{align*}
and since $\zeta(2)=\pi^2/6$, it is clear that the total contribution of $n\neq 0$ is bounded above by $0.6$.
Also, we have $[x^1]1/(2-e^x)=[x^0]1/(2-e^x)=1$,
and therefore monotonicity for $k\geq 0$ follows.

Next, this monotonicity implies that $\left[x^k\right]e^{-x}/(2-e^x)$ is nonnegative and bounded by $\left[x^k\right]1/(2-e^x)$:
expanding the Cauchy product $e^{-x}\cdot (2-e^x)^{-1}$, we see that the coefficients are given by an alternating sum of nonincreasing values, which immediately implies the claim.

Using this nonnegativity property, and also the fact that $e^x$ has nonnegative coefficients, we obtain
\begin{align*}
\left\lvert
\left[x^k\right]
\frac {e^{-x}}{2-e^x}
\biggl(1-
\left(\frac{e^x}2\right)^{m_0-1}
\biggr)
\right\rvert
&\leq
\left[x^k\right]
\frac {e^{-x}}{2-e^x}
\leq 
\left[x^k\right]
\frac1{2-e^x}
\leq \frac1{(\log 2)^k}.
\end{align*}
%
%\lvert
%\left[x^k\right]
%e^x
%\sum_{0\leq k\leq m_0-2}
%\left(e^{-x}/2\right)^k
%\right\rvert
%\\
%&\leq
%\left[x^k\right]
%\sum_{1\leq k\leq m_0-1}
%\left(e^x/2\right)^k
%\leq
%\left[x^k\right]
%\frac 1{1-e^x/2}
%\leq \frac{2}{(\log 2)^k}.
%\end{align*}
%The same is true for $M_1$, that is, $e^{\pm x}$ replaced by $e^{\mp x}$.
Moreover, since $C$ has nonnegative coefficients bounded by the coefficients of $e^x$, we have
\begin{align*}
\left\lvert
\left[x^k\right]\left(e^x/2\right)^{m_0-1}C(x)
\right\rvert
\leq 
2\left[x^k\right]\left(e^x/2\right)^{m_0}
\leq \frac {2}{\sqrt{2\pi k}}\frac 1{(\log 2)^k}.
\end{align*}
The same is true for $S$ in place of $C$.
It follows that the coefficients of the entries of $M_0^m$ and $M_1^m$ are bounded by 
\[\frac 12+\frac1{\sqrt{2\pi k}(\log 2)^k}+\frac1{2(\log 2)^k}\leq \frac 2{(\log 2)^k}\] for $k\geq 1$. This finishes the proof of Lemma~\ref{lem_higher_coeffs}.
\end{proof}
%This proof also shows that
%\begin{equation}\label{eqn_moment_1}
%\bigl[x^k\bigr]\frac{e^x}{2-e^{-x}}\leq 2(\log 2)^{-k}
%\end{equation}
%for $k\geq 0$,
In particular, using~\eqref{eqn_M1}, it follows from this proof that
\begin{equation}\label{eqn_start}
a_k(0)\leq (\log 2)^{-k}\quad\mbox{and}\quad\lvert b_k(0)\rvert\leq (\log 2)^{-k}\quad\mbox{for }k\geq 1.
\end{equation}
%for $k\geq 1$.

We are now prepared to prove upper bounds for the moments $m_k(t)$.
\begin{proposition}\label{prp_moment_bounds}
Set $A_k=2\cdot(3/2)^{k-1}/k!$ for $k\geq 1$.
There exist constants $B_k$, $C_k$, and $E_k$ (for $k\geq 1$) and $D_k$ (for $k\geq 2$) such that for all $r\geq 1$, and all $t\geq 1$ having $r$ blocks we have
\[
\begin{array}{r@{\hskip 1mm}c@{\hskip 1mm}lcr@{\hskip 1mm}c@{\hskip 1mm}lcr@{\hskip 1mm}c@{\hskip 1mm}l}
\multicolumn{3}{c}{k=0}&\hphantom{\hskip 3em}&
\multicolumn{3}{c}{k=1}&\hphantom{\hskip 3em}&
\multicolumn{3}{c}{k\geq 2}\\
\hline\vspace{-2mm}\\
\left\lvert a_0(t)\right\rvert&=&2;&&
\left\lvert a_2(t)\right\rvert&\leq&A_1r+B_1;&&
\left\lvert a_{2k}(t)\right\rvert&\leq& A_kr^k+B_kr^{k-1};\\[2mm]
\left\lvert a_1(t)\right\rvert&=&0;&&
\left\lvert a_3(t)\right\rvert&\leq&C_1r;&&
\left\lvert a_{2k+1}(t)\right\rvert&\leq&C_kr^k;\\[2mm]
\left\lvert b_0(t)\right\rvert&=&0;&&
\left\lvert b_2(t)\right\rvert&\leq&1;&&
\left\lvert b_{2k}(t)\right\rvert&\leq&A_{k-1}r^{k-1}+D_kr^{k-2};\\[2mm]
\left\lvert b_1(t)\right\rvert&=&0;&&
\left\lvert b_3(t)\right\rvert&\leq& E_1;&&
\left\lvert b_{2k+1}(t)\right\rvert&\leq& E_kr^{k-1}.
\end{array}
\]
\end{proposition}

\begin{proof}
%By induction, it is easy to see that
We proceed by induction on $k$,
using Lemmas~\ref{lem_lower_coeffs} and~\ref{lem_higher_coeffs}.
Clearly, the statement is true for $k=0$, since $m_0(t)=1$ and $m_1(t)=0$ for all $t\geq 0$.

We begin with the treatment of the even moments.
%Set $A_1=2$ and $B_1=0$
The first step is to verify the following identities for the case $k=1$:
\begin{equation}\label{eqn_core_estimates_even_1}
\begin{array}{l@{\hskip 2mm}c@{\hskip 2mm}c@{\hskip 1mm}c@{\hskip 1mm}c@{\hskip 1mm}c@{\hskip 1mm}c@{\hskip 1mm}c@{\hskip 1mm}c}
\displaystyle
a_{2}(2^mt)&=&a_{2}(t)&+&\displaystyle\frac{2^m-1}{2^m}&+&\displaystyle
\frac{2^m-1}{2^m}&b_{2}(t);\\[4mm]
%\label{eqn_a_zeros}\\
b_{2}(2^mt)&=&&-&\displaystyle\frac{2^m-1}{2^m}&+&\displaystyle
\frac1{2^m}&b_{2}(t);\\[4mm]
%\label{eqn_b_zeros}\\
a_{2}(2^mt+2^m-1)&=&a_{2}(t)&+&\displaystyle\frac{2^m-1}{2^m}&-&\displaystyle\frac{2^m-1}{2^m}&b_{2}(t);\\[4mm]
%\label{eqn_a_ones}\\
b_{2}(2^mt+2^m-1)&=&&&\displaystyle\frac{2^m-1}{2^m}&+&\displaystyle\frac{1}{2^m}&b_{2}(t).
%\label{eqn_b_ones}
\end{array}
\end{equation}
In the induction step, we will make use of the following identities, valid for $k\geq 2$: we have
\begin{equation}\label{eqn_core_estimates_even}
\begin{array}{l@{\hskip 2mm}c@{\hskip 2mm}c@{\hskip 1mm}c@{\hskip 1mm}c@{\hskip 1mm}c@{\hskip 1mm}c@{\hskip 1mm}c@{\hskip 1mm}c@{\hskip 1mm}c@{\hskip 1mm}c}
\displaystyle
a_{2k}(2^mt)&=&a_{2k}(t)&+&\displaystyle\frac{2^m-1}{2^{m+1}}&a_{2k-2}(t)&+&\displaystyle
\frac{2^m-1}{2^m}&b_{2k}(t)&+&\LandauO(r^{k-2});\\[4mm]
b_{2k}(2^mt)&=&&-&\displaystyle\frac{2^m-1}{2^{m+1}}&a_{2k-2}(t)&+&\displaystyle
\frac1{2^m}&b_{2k}(t)&+&
\LandauO(r^{k-2});\\[4mm]
a_{2k}(2^mt+2^m-1)&=&a_{2k}(t)&+&\displaystyle\frac{2^m-1}{2^{m+1}}&a_{2k-2}(t)&-&\displaystyle\frac{2^m-1}{2^m}&b_{2k}(t)&+&\LandauO(r^{k-2});\\[4mm]
b_{2k}(2^mt+2^m-1)&=&&&\displaystyle\frac{2^m-1}{2^{m+1}}&a_{2k-2}(t)&+&\displaystyle\frac{1}{2^m}&b_{2k}(t)&+&\LandauO(r^{k-2}),
\end{array}
\end{equation}
where the implied constants only depend on $k$.
It is notable that only even moments are involved in these identities!

\noindent
\textbf{Proof of~\eqref{eqn_core_estimates_even_1} and~\eqref{eqn_core_estimates_even}.}
We begin with the first and third lines of~\eqref{eqn_core_estimates_even_1} and~\eqref{eqn_core_estimates_even}.
Appending a block of $\tO$s or $\tL$s of length $m$ to $t$, we obtain $t'$;
by~\eqref{eqn_momentcoeff_rec} we have
\begin{equation*}
\begin{aligned}
\left(\begin{matrix}F_{t'}(x)\\G_{t'}(x)\end{matrix}\right)
&=M_i^m
\left(\begin{matrix}F_t(x)\\G_t(x)\end{matrix}\right),
\end{aligned}
\end{equation*}
where $i\in\{0,1\}$.
Using Lemma~\ref{lem_lower_coeffs} and expanding Cauchy products, we obtain 
\begin{equation}\label{eqn_a2k_expanded}
\begin{aligned}
a_{2k}(t')=a_{2k}(t)&+\frac{2^m-1}{2^{m+1}}a_{2k-2}+
\sum_{3\leq \ell\leq 2k}a_{2k-\ell}(t) \left[x^{\ell}\right]\mathfrak a_0 (x)
\\&\pm\frac{2^m-1}{2^m}b_{2k}(t)+\sum_{1\leq \ell\leq 2k}b_{2k-\ell}(t)\left[x^\ell\right]\mathfrak b_0(x).
\end{aligned}
\end{equation}
Here, the ``$+$''-part corresponds to appending a block of $\tO$s and the ``$-$''-part to appending a block of $\tL$s.
.
Clearly, for $k=1$ we have
\[
\sum_{3\leq \ell\leq 2k}a_{2k-\ell}(t) \left[x^{\ell}\right]\mathfrak a_0 (x)=
\sum_{1\leq \ell\leq 2k}b_{2k-\ell}(t)\left[x^\ell\right]\mathfrak b_0(x)=
0,\]
which implies the statements concerning $a_2$ in~\eqref{eqn_core_estimates_even_1}.

Inserting the estimates for the coefficients of $M_i^m$(Lemma~\ref{lem_higher_coeffs}) 
and the induction hypothesis, we obtain for $k\geq 2$, using $a_0(t)=2$,
\begin{align*}
%\hspace{2em}&\hspace{-2em}
S_1&\coloneqq
\sum_{3\leq \ell\leq 2k}a_{2k-\ell}(t) \left[x^{\ell}\right]\mathfrak a_0 (x)
\leq 
\sum_{\substack{4\leq \ell\leq 2k\\2\mid \ell}}
a_{2k-\ell}(t)\frac2{(\log 2)^\ell}
+
\sum_{\substack{3\leq \ell\leq 2k-1\\2\nmid \ell}}
a_{2k-\ell}(t)\frac2{(\log 2)^\ell}
\\&
\leq 
4\bigl(\log 2\bigr)^{-2k}+
\sum_{2\leq j<k}
\left(A_{k-j}r^{k-j}+B_{k-j}r^{k-j-1}\right)\frac2{(\log 2)^{2j}}
+
\sum_{2\leq j<k}
C_{k-j}\frac{2r^{k-j}}{(\log 2)^{2j-1}}.
\end{align*}
The sums are identically zero if $r=0$ or $k<3$;
for the other cases, we note that for all integers $m\geq 1$,
\begin{equation}\label{eqn_log_sum}
\sum_{0\leq j<m}(\log 2)^{-2j}=\frac{(\log 2)^{-2m}-1}{(\log 2)^{-2}-1}\leq (\log 2)^{-2m}
\end{equation}
since the appearing denominator is greater than $1$.
Therefore
\begin{align*}
\sum_{2\leq j<k}A_{k-j}\frac{r^{k-j}}{(\log 2)^{2j}}
&\leq
\max_{1\leq j\leq k-2}A_j
\frac{r^{k-2}}{(\log 2)^4}
\sum_{0\leq \ell<k-2}\frac 1{(\log 2)^{2\ell}}
\leq 
\frac{r^{k-2}}{(\log 2)^{2k}}
\max_{1\leq j\leq k-2}A_j,
\end{align*}
which is also valid for $A$ replaced by $B$ resp. $C$.
We obtain for $k\geq 2$
\[S_1\leq d_1^{(1)}(k) r^{k-2},\]
where
\[d_1^{(1)}(k)=
2\bigl(\log 2\bigr)^{-2k}
\left(2+\max_{1\leq j\leq k-2}A_j+\max_{1\leq j\leq k-2}B_j+\max_{1\leq j\leq k-2}C_j\right).\]
Here and in the following the maximum over the empty index set is defined as $0$.

In an analogous fashion, we treat the second sum in~\eqref{eqn_a2k_expanded}, which is nonzero only if $k\geq 2$.
%For $k=1$ we have
%\[
%\sum_{1\leq \ell\leq 2k}b_{2k-\ell}(t) \left[x^{\ell}\right]\mathfrak b_0 (x)=0,\]
We use the hypothesis $\lvert b_2(t)\rvert\leq 1$, and obtain for $k\geq 2$
%some constant $d_1^{(2)}$ such that 
\begin{align*}
\hspace{2em}&\hspace{-2em}
S_2\coloneqq\sum_{1\leq \ell\leq 2k}b_{2k-\ell}(t) \left[x^{\ell}\right]\mathfrak b_0 (x)
\leq 
\sum_{\substack{2\leq \ell\leq 2k\\2\mid \ell}}
b_{2k-\ell}(t)\frac2{(\log 2)^\ell}
+
\sum_{\substack{1\leq \ell\leq 2k-1\\2\nmid \ell}}
b_{2k-\ell}(t)\frac2{(\log 2)^\ell}
\\&\leq
\sum_{1\leq j\leq k-2}
\left(A_{k-j-1}r^{k-j-1}+D_{k-j}r^{k-j-2}\right)\frac2{(\log 2)^{2j}}
+2\bigl(\log 2\bigr)^{-2k+2}
\\&+
\sum_{1\leq j\leq k-1}
E_{k-j}r^{k-j-1}\frac2{(\log 2)^{2j-1}}.
\end{align*}
Similarly to the treatment of $S_1$, using~\eqref{eqn_log_sum}, we have for $k\geq 2$
\[S_2\leq d_1^{(2)}(k) r^{k-2},\]
where
\[d_1^{(2)}(k)=
2(\log 2)^{-2k}
%          -2k+1 is sufficient
\left(1+\max_{1\leq j\leq k-2}A_j+\max_{2\leq j\leq k-1}D_j+\max_{1\leq j\leq k-1}E_j\right).
\]

This implies the statements for $a_{2k}$ in~\eqref{eqn_core_estimates_even}.
We proceed to $b_{2k}$ and obtain
\begin{equation}\label{eqn_b2k_step}
\begin{aligned}
b_{2k}(t')=&\mp \frac{2^m-1}{2^{m+1}}a_{2k-2}(t)+
\sum_{3\leq \ell\leq 2k}a_{2k-\ell}(t) \left[x^{\ell}\right]\mathfrak c_0 (x)
\\&+\frac1{2^m}b_{2k}(t)+\sum_{1\leq \ell\leq 2k}b_{2k-\ell}(t)\left[x^\ell\right]\mathfrak d_0(x).
\end{aligned}
\end{equation}
Again, for $k=1$, the sums vanish, and we obtain the second and fourth lines of~\eqref{eqn_core_estimates_even_1}.

For $k\geq 2$, the two sums occurring in~\eqref{eqn_b2k_step} can be estimated by $d^{(1)}_1(k)r^{k-2}$ and $d^{(2)}_2(k)r^{k-2}$ respectively.
This follows by replacing $\mathfrak a$ by $\mathfrak c$ and $\mathfrak b$ by $\mathfrak d$ and recycling the argument from above.
%
%get an additional error term $\LandauO(r^{k-2})$ by the induction hypothesis,
This implies lines two and four of~\eqref{eqn_core_estimates_even}.

%In a similar manner, the other identities can be obtained.

It follows that $d_1(k)=d^{(1)}_1(k)+d^{(2)}_1(k)$ is an admissible constant for all of the four formulas in~\eqref{eqn_core_estimates_even}.

\noindent
\textbf{Deriving bounds for the even moments.}
We apply the eight equations in~\eqref{eqn_core_estimates_even_1} and~\eqref{eqn_core_estimates_even} successively in order to obtain the ``even part'' of the statement, that is, the estimates for $a_{2k}(t)$ and $b_{2k}(t)$.

We begin with $b_2$ and show $\lvert b_2(t)\rvert\leq 1$ by induction: we have $b_2(0)=m_2(0)-m_2(1)=-1$,
and by~\eqref{eqn_core_estimates_even_1} we obtain
\[\lvert b_2(t')\rvert\leq \frac{2^m-1}{2^m}+\frac 1{2^m}\leq 1\]
for both $t'=2^mt$ and $t'=2^mt+2^m-1$.
Moreover,
$a_2(0)=1$ and 
\[a_2(t')=
a_2(t)+\frac{2^m-1}{2^m}\pm \frac{2^m-1}{2^m}b_2(t)
\leq a_2(t)+2,\]
which implies $a_2(t)\leq 2r+1$ for all $t\geq 0$ having $r$ blocks.
We therefore set $A_1=2$ and $B_1=1$.
Clearly $\lvert a_2(t)\rvert\leq A_1r^1+B_1r^0$,
and the estimates for $a_2$ and $b_2$ in the proposition are proved.

We assume now that $k\geq 2$ and we consider
$b_{2k}$: we have
\[
\left\lvert b_{2k}(t')\right\rvert
\leq
\left\lvert
\mp\frac{2^m-1}{2^{m+1}}a_{2k-2}(t)+\frac1{2^m}b_{2k}
\right\rvert
+d_1(k)r^{k-2}
\leq \frac {\left\lvert b_{2k}(t)\right\rvert+a_{2k-2}(t)}{2}+d_1(k) r^{k-2},
\]
where $t'$ results from $t$ by appending $\tO$s (if $t$ is odd) or by appending $\tL$s (if $t$ is even).

Here $r$ is the number of blocks in $t$.
%For odd $t$, we iterate this twice in order to get
%this even/odd business is annoying but necessary.
By iteration, exploiting the denominator $2$ (geometric series!), and by applying the induction hypothesis and
\[\left\lvert b_{2k}(0)\right\rvert\leq \frac 1{(\log 2)^{2k}},\]
we obtain
\[
\left\lvert b_{2k}(t)\right\rvert
\leq A_{k-1}r^{k-1}+B_{k-1}r^{k-2}+2d_1(k)r^{k-2}+\frac1{(\log 2)^{2k}}
\]
if $t$ has $r$ blocks.
We therefore set $D_k=B_{k-1}+2d_1(k)+(\log 2)^{-2k}$ and this case is completed.
We proceed to the case $a_{2k}$, where $k\geq 2$:
each time we append a block of $\tO$s or $\tL$s to $t$, we add at most
%at this point the loss happens. for a future research work, we have to analyze the matter in more detail.
\begin{align*}
\left\lvert a_{2k}(t')-a_{2k}(t)\right\rvert
&\leq
\frac{2^m-1}{2^{m+1}}a_{2k-2}(t)+\frac{2^m-1}{2^m}\lvert b_{2k}(t)\rvert
+d_1(k)r^{k-2}
\\&\leq
\frac 12a_{2k-2}(t)+\left\lvert b_{2k}(t)\right\rvert+d_1(k)r^{k-2}
\\&\leq \frac 32 A_{k-1}r^{k-1}+
\left(B_{k-1}+D_k+d_1(k)\right)
%B_k
r^{k-2}.
\end{align*}
This follows from~\eqref{eqn_core_estimates_even} and line three of the induction statement.
We wish to successively append a block of $\tO$s or $\tL$s; this corresponds to summing this inequality in $r$. 
For $\ell\geq 1$ and $N\geq 0$ we have
\[\sum_{1\leq n<N}n^{\ell-1}
\leq
\int_1^{N} n^{\ell-1}\ud n
\leq \frac{N^\ell}\ell.\]
Noting that
$a_{2k}(0)=m_{2k}(1)\leq (\log 2)^{-2k}$ for $k\geq 1$ by~\eqref{eqn_start},
we obtain therefore
\[a_{2k}(t)\leq A_kr^k+B_kr^{k-1}\]
with 
\begin{equation}\label{eqn_AB_definition}
\begin{aligned}
A_k&=\frac {3}{2k}A_{k-1}\quad\mbox{and}\\
B_k&=\frac 1{k-1}\bigl(B_{k-1}+D_k+d_1(k)\bigr)+\frac1{(\log 2)^{2k}}
\\&=\frac 1{k-1}\bigl(2B_{k-1}+3d_1(k)\bigr)+\frac2{(\log 2)^{2k}},
\end{aligned}
\end{equation}
which proves the part of the induction statement concerning the even moments.

\noindent
\textbf{Deriving bounds for the odd moments.}
For the odd case, concerning $a_{2k+1}$ and $b_{2k+1}$, we proceed similarly.
%only have to prove rougher bounds; we are not concerned with the values $A_k$.
Suppose that $k\geq 1$.
Using~\eqref{eqn_momentcoeff_rec}, Lemma~\ref{lem_lower_coeffs} and expanding Cauchy products, we obtain 
\begin{align*}
a_{2k+1}(t')=a_{2k+1}(t)&+\sum_{2\leq \ell\leq 2k+1}a_{2k+1-\ell}(t) \left[x^{\ell}\right]\mathfrak a_0 (x)
\\&\pm\frac{2^m-1}{2^m}b_{2k+1}(t)+\sum_{1\leq \ell\leq 2k+1}b_{2k+1-\ell}(t)\left[x^\ell\right]\mathfrak b_0(x),
\end{align*}
where ``$+$'' corresponds to appending a block of $\tO$s.
By Lemma~\ref{lem_higher_coeffs} and the induction hypothesis, using $a_1(t)=0$ and $a_0(t)=2$, we have for $k\geq 1$
\begin{align*}
\hspace{2em}&\hspace{-2em}
\sum_{2\leq \ell\leq 2k+1}a_{2k+1-\ell}(t) \left[x^{\ell}\right]\mathfrak a_0 (x)
\leq
\sum_{\substack{2\leq \ell\leq 2k\\2\mid \ell}}
a_{2k+1-\ell}(t)\frac2{(\log 2)^\ell}
+
\sum_{\substack{3\leq \ell\leq 2k+1\\2\nmid \ell}}
a_{2k+1-\ell}(t)\frac2{(\log 2)^\ell}
\\&\leq
\sum_{1\leq j<k}C_{k-j}r^{k-j}\frac2{(\log 2)^{2j}}
+
\sum_{1\leq j<k}\left(A_{k-j}r^{k-j}+B_{k-j}r^{k-j-1}\right)
\frac2{(\log 2)^{2j+1}}
\\&+
\frac{4}{(\log 2)^{2k+1}}
\leq
d_2^{(1)}(k)r^{k-1},
\end{align*}
where
\[
d_2^{(1)}(k)=
2\left(\log 2\right)^{-2k-1}
\left(2+\max_{1\leq j\leq k-1}A_j
+\max_{1\leq j\leq k-1}B_j+\max_{1\leq j\leq k-1}C_j\right).
\]
Moreover, using also the even case proved above and $\lvert b_2(t)\rvert\leq 1$, we get
\begin{align*}
\hspace{4em}&\hspace{-4em}
\sum_{1\leq \ell\leq 2k+1}b_{2k+1-\ell}(t) \left[x^{\ell}\right]\mathfrak b_0 (x)
\leq
\sum_{\substack{2\leq \ell\leq 2k\\2\mid \ell}}
b_{2k+1-\ell}(t)\frac2{(\log 2)^\ell}
+
\sum_{\substack{1\leq \ell\leq 2k+1\\2\nmid \ell}}
b_{2k+1-\ell}(t)\frac2{(\log 2)^\ell}
\\&\leq
\sum_{1\leq j<k}
E_{k-j}r^{k-j-1}
\frac2{(\log 2)^{2j}}
\\&+
\sum_{1\leq j<k}\left(
A_{k-j}r^{k-j}+D_{k-j+1}r^{k-j-1}\right)
\frac2{(\log 2)^{2j-1}}
+
\frac2{(\log 2)^{2k-1}}
\\&\leq
d_2^{(2)}(k)r^{k-1},
\end{align*}
where
\[
d_2^{(2)}(k)=
2\left(\log 2\right)^{-2k-1}
\left(1+\max_{1\leq j\leq k-1}A_j+\max_{2\leq j\leq k}D_j+\max_{1\leq j\leq k-1}E_j\right).
\]
This estimate is valid for $k\geq 1$.
Therefore
\begin{align}\label{eqn_a2k1}
a_{2k+1}(t')=a_{2k+1}(t)\pm\frac{2^m-1}{2^m}b_{2k+1}(t)+\LandauO_1\bigl(d_2(k)r^{k-1}\bigr),
\end{align}
where $r$ is the number of blocks in $t$
and $d_2(k)=d_2^{(1)}(k)+d_2^{(2)}(k)$,
and where we use the symbol $\LandauO_1$ to indicate that the implied constant is bounded by $1$.

Concerning $b_{2k+1}$, we obtain from Lemma~\ref{lem_lower_coeffs}
\begin{align*}
b_{2k+1}(t')&=\sum_{2\leq \ell\leq 2k+1}a_{2k+1-\ell}(t) \left[x^{\ell}\right]\mathfrak c_0 (x)
+\frac1{2^m}b_{2k+1}(t)+\sum_{1\leq \ell\leq 2k+1}b_{2k+1-\ell}(t)\left[x^\ell\right]\mathfrak d_0(x),
\end{align*}
and therefore by the above argument (replacing $\mathfrak a$ and $\mathfrak b$ by $\mathfrak c$ and $\mathfrak d$ respectively)

\begin{align}\label{eqn_b2k1}
b_{2k+1}(t')&=\frac1{2^m}b_{2k+1}(t)+\LandauO_1\bigl(d_2(k)r^{k-1}\bigr).
\end{align}
By repeated application of this identity,
using~\eqref{eqn_start}, we obtain
\[
\lvert b_{2k+1}(t)\rvert \leq \frac 12\lvert b_{2k+1}(0)\rvert+2d_2(k)r^{k-1}
\leq
E_kr^{k-1}
\]
for all $t$, where $E_k=2d_2(k)+(\log 2)^{-2k-1}$ for $k\geq 1$.
%Applying~\eqref{eqn_b2k1} recursively, we obtain
%%
%$b_{2k+1}(t)=\LandauO(r^{k-1})$ for all $t$ and all $k\geq 1$.
Inserting this into~\eqref{eqn_a2k1}, we obtain
\[\lvert a_{2k+1}(t')\rvert \leq
\lvert a_{2k+1}(t)\rvert +\bigl(3d_2(k)+(\log 2)^{-2k-1}\bigr)r^{k-1}\]
and therefore by summation, using~\eqref{eqn_start} again,
\[a_{2k+1}(t)\leq C_kr^k,\]
where $C_k=3d_2(k)/k+2(\log 2)^{-2k-1}$.
This is valid for all $k\geq 1$.
The proof is complete.
\end{proof}
Summarizing, the recurrence for the quantities $A_j$ through $E_j$, used in the proof, is as follows:
\[
\begin{array}{r@{\hskip 1mm}c@{\hskip 1mm}lcr@{\hskip 1mm}c@{\hskip 1mm}l}
\multicolumn{3}{c}{k=1}&\hphantom{\hskip 3em}&
\multicolumn{3}{c}{k\geq 2}\\
\hline\vspace{-2mm}\\
A_1&=&2;&&
A_k&=&3A_{k-1}/(2k);
\\[2mm]
B_1&=&0;&&
B_k&=&\bigl(2B_{k-1}+3d_1(k)\bigr)/(k-1)+2(\log 2)^{-2k};
\\[2mm]
C_1&=&3d_2(1);&&
C_k&=&3d_2(k)/k+2(\log 2)^{-2k-1};
\\[2mm]
&&&&D_k&=&B_{k-1}+2d_1(k)+(\log 2)^{-2k};
\\[2mm]
E_1&=&2d_2(1);&&
E_k&=&2d_2(k)+(\log 2)^{-2k-1},
\end{array}
\]
where
\begin{multline*}
d_1(k)=
\frac2{(\log 2)^{2k}}
\biggl(3+2\max_{1\leq j\leq k-2}A_j+\max_{1\leq j\leq k-2}B_j\\+\max_{1\leq j\leq k-2}C_j+\max_{2\leq j\leq k-1}D_j+\max_{1\leq j\leq k-1}E_j\biggr)
\end{multline*}
for $k\geq 2$ and
\begin{multline*}d_2(k)=
\frac2{(\log 2)^{2k+1}}
\biggl(3+2\max_{1\leq j\leq k-1}A_j+\max_{1\leq j\leq k-1}B_j\\+\max_{1\leq j\leq k-1}C_j+\max_{2\leq j\leq k}D_j+\max_{1\leq j\leq k-1}E_j\biggr)
\end{multline*}
for $k\geq 1$.
Using this recurrence, it is easy to compute explicit bounds for the values $A_j$ through $E_j$, in particular, choosing $\varepsilon<1/4$, this leads to an effective bound $r_0$ such that $c_t>1/4$ as soon as $t$ has at least $r_0$ blocks.
However, we do not believe that these numerical values are particularly enlightening (and fairly large).
We therefore limit ourselves to a short summary:
for $k=1$, we see that $d_2(1)=6(\log 2)^{-4}$, from which we obtain $C_1$ and $E_1$.
In the step $k-1\rightarrow k$, we have to compute $d_1(k)$ first; then $B_k$ and $D_k$ can be obtained, and $d_2(k)$ as the next step (note that for the maximum $\max_{2\leq j\leq k}D_j$ we need $D_k$).
Finally, $C_k$ and $E_k$ can be computed.

%}}}

\subsection{Bounding the characteristic function using a matrix product}
The correlations $\gamma_t(\vartheta)$ satisfy the following recurrence (see B\'esineau~\cite{B1972}): for all $t\geq 0$,
\[\gamma_0(\vartheta)=1,\quad \gamma_{2t}(\vartheta)=\gamma_t(\vartheta),\quad\gamma_{2t+1}(\vartheta)=\frac{\e(\vartheta)}2\gamma_t(\vartheta)+\frac{\e(-\vartheta)}2\gamma_{t+1}(\vartheta).\]
In order to capture this using a matrix product, we define
\[A(0)=\left(\begin{matrix}1&0\\\e(\vartheta)/2&\e(-\vartheta)/2\end{matrix}\right),\quad
A(1)=\left(\begin{matrix}\e(\vartheta)/2&\e(-\vartheta)/2\\0&1\end{matrix}\right).
\]
In~\cite{MS2012} we used the representation
\[
\gamma_t(\vartheta)
=
\left(\begin{matrix}1&0\end{matrix}\right)
A(\varepsilon_0)\cdots A(\varepsilon_\nu)
\left(\begin{matrix}1\\u\end{matrix}\right),
\]
where
$t=(\varepsilon_\nu\ldots\varepsilon_0)_2$ and
$u=\gamma_1(\vartheta)=\e(\vartheta)/(2-\e(-\vartheta))$.

Using this matrix identity, we proved~\cite[Lemma~2]{S2019} an upper bound for $\gamma_t(\vartheta)$ depending on the number $r$ of blocks of $\tL$s occurring in $t$. By a slight variation (we handled the values $\omega_t(\vartheta)=\gamma_t(\vartheta)/u$, where $u$ is bounded by $1$ in absolute value) we obtain the following statement.
\begin{lemma}\label{lem_muntjak}
Assume that $t\geq 1$ contains at least $4M+1$ blocks.
Then
\[\left\lvert \gamma_t(\vartheta)\right\rvert\leq \left(1-\frac 12\lVert \vartheta\rVert^2\right)^M.\]
\end{lemma}

\subsection{Splitting the integral}

We are interested in bounding the integral
\[
\int_0^{1/2}
\imagpart\gamma_t(\vartheta)\cot(\pi \vartheta) \ud \vartheta
\]
by $\varepsilon$.
(Note that the integrand is an even function and therefore the integrals over $[0,1/2]$ and $[1/2,1]$ yield the same value.)
We split the integration at the point 
$\vartheta_0=r^{-1/2}R$,
where the integer $R$ is chosen in a moment and $r$ is the number of blocks in $t$.

We begin with the estimation of the right part of the integral.
Let $M$ be maximal such that $4M+1\leq r$.
Then by Lemma~\ref{lem_muntjak}, we have
\[\bigl\lvert\gamma_t(\vartheta)\bigr\rvert
\leq \left(1-\frac 12\left\lVert \vartheta\right\rVert^2\right)^M
\leq \exp\left(-\frac M2\left\lVert \vartheta\right\rVert^2\right).
\]
We have $(r-4)/4\leq M\leq (r-1)/4$ (because $(r-5)/4\geq M$ is impossible due to the maximality of $M$, therefore $4M+5>r$, which implies $4M+4\geq r$).
Also, for $0\leq x\leq \pi/2$, we have the elementary inequality
\begin{equation}\label{eqn_cot_estimate}
\cot x\leq 1/x.
\end{equation}
We obtain 
\begin{multline*}
\int_{\vartheta_0}^{1/2}\imagpart \gamma_t(\vartheta)\cot(\pi\vartheta)\ud\vartheta
\leq
\int_{\vartheta_0}^{1/2}
\frac 1\vartheta
\exp\left(-\frac M2\left\lVert \vartheta\right\rVert^2\right)
\ud\vartheta
\\\leq
\sum_{m=R}^{\infty}
\int_{mr^{-1/2}}^{(m+1)r^{-1/2}}
\frac 1\vartheta
\exp\left(-\frac M2\vartheta^2\right)
\ud\vartheta
\leq
r^{-1/2}
\sum_{m=R}^{\infty}
\frac 1{mr^{-1/2}}
\exp\left(-\frac M2m^2/r\right)
\\\leq
\sum_{m=R}^{\infty}
\frac 1m
\exp\left(-\frac {m^2}8\left(1-\frac 4r\right)\right)
\leq
\sum_{m=R}^{\infty}
\frac 1m
\exp\left(-\frac {m^2}{16}\right).
\end{multline*}
This is valid for $r\geq 8$.
Since $\exp(-m^2/16)=O(1/m)$ for $m\rightarrow\infty$, this infinite sum is bounded by $c/R$ for some absolute (effective) constant $c$.
We therefore choose the integer $R=R(\varepsilon)$ large enough such that $c/R\leq \varepsilon/3$,
and we obtain
\begin{equation}\label{eqn_estimate_right_part}
\int_{\vartheta_0}^{1/2}\imagpart \gamma_t(\vartheta)\cot(\pi\vartheta)\ud\vartheta
\leq \frac{\varepsilon}{3}.
\end{equation}

The left part of the integral will be estimated using upper bounds for the odd moments, which is Proposition~\ref{prp_moment_bounds}.

From the estimate of $a_{2k+1}(t)$ and $b_{2k+1}$ in Proposition~\ref{prp_moment_bounds} we get by the triangle inequality
\[\left\lvert m_{2k+1}(t)\right\rvert\leq E_k' r^k\]
for some constant $E_k'$ only depending on $k$.
Moreover, from the estimate for $a_{2k}(t)$ we obtain by nonnegativity of the even moments %and $b_{2k}(t)$ we obtain
\[m_{2K}(t)\leq A_{K} r^K+B_{K-1} r^{K-1}\]
and 
\[m_{2K+2}(t)\leq A_{K+1} r^{K+1}+B_K r^K.\]
For $r$ greater than some $r_0(K)$ we therefore have
\begin{equation}\label{eqn_even_moments_estimate}
m_{2K}(t)\leq 2A_Kr^K\mbox{ and }m_{2K+2}(t)\leq 2A_{K+1}r^{K+1}
\end{equation}
for all $t$ having at least $r$ blocks.
Let $K$ be large enough so that
\begin{equation}\label{eqn_LK_estimate}
L_K=
2\frac{\sqrt{(2K)!(2K+2)!A_KA_{K+1}}(2\pi)^{2K+1}}{(2K+1)!}\leq \frac{\varepsilon/3}{R^{2K+1}}.
\end{equation}
Note that for this inequality, we use the factor $k!$ in the denominator of $A_k$ in an essential way!
By~\eqref{eqn_imagpart_moments} and~\eqref{eqn_even_moments_estimate} we obtain for $r\geq r_0(K)$ %and $\lVert \vartheta\rVert\leq \vartheta_0=r^{-1/2}R$
\begin{align*}
\bigl\lvert\imagpart \gamma_t(\vartheta)\bigr\rvert
\leq \sum_{0\leq k<K}(2\pi \vartheta)^{2k+1}E_k'r^k
+L_K\vartheta^{2K+1}r^{K+1/2}.
\end{align*}
%with an implied constant depending only on $K$.
Using~\eqref{eqn_cot_estimate}, we obtain for $\vartheta\leq \vartheta_0$
\begin{align*}
\bigl\lvert\imagpart \gamma_t(\vartheta)
\cot(\pi \vartheta)
\bigr\rvert
&\leq \sum_{0\leq k<K}(2\pi)^{2k+1}E_k'\vartheta_0^{2k}r^k
+L_K\vartheta_0^{2K}r^{K+1/2}
\\&=
\sum_{0\leq k<K}(2\pi)^{2k+1}E_k'R^{2k}
+L_KR^{2K}r^{1/2}.
\end{align*}
Integrating from $0$ to $\vartheta_0$ yields for $r\geq r_0(K)$
\[\int_0^{\vartheta_0}
\imagpart \gamma_t(\vartheta)\cot(\pi\vartheta)\ud\vartheta
\leq
r^{-1/2}\sum_{0\leq k<K}(2\pi R)^{2k+1}E_k'
+L_KR^{2K+1}.
\]
The second summand is bounded by $\varepsilon/3$, using~\eqref{eqn_LK_estimate}.
The sum over $k$ does not depend on $r$.
For $r\geq r_1(K)$, we therefore have
\[\int_0^{\vartheta_0}
\imagpart \gamma_t(\vartheta)\cot(\pi\vartheta)\ud\vartheta\leq \frac{2\varepsilon}3
\]
and the theorem is proved.
\qed

%\section{Future directions}
\begin{remark}
We plan to prove more detailed estimates for the moments $m_k$ in the future.
For example, it is known~\cite{EH2018} that $m_{2k}(t)$ is \emph{usually} of size $r^k/(2^kk!)$, where $r$ is the number of blocks in $t$ (we skip the precise formulation of the property proved in~\cite{EH2018}).
We wish to sharpen this estimate, and consequently prove a \emph{lower bound} for the values $\delta(0,t)=\int_0^1\realpart \gamma_t(\vartheta)\ud\vartheta$ for $T-\LandauO\left(T^{1-\varepsilon}\right)$ many $t<T$.
Using~\eqref{eqn_ct_rep}, and also improving the estimates of the odd moments, we hope to obtain $c_t>1/2$ for these $t$ in this way, thus significantly improving the error term in the result~\cite{DKS2016} by Drmota, Kauers, and the author.
\end{remark}

\subsection*{Acknowledgements.}
The author is thankful to Johannes Morgenbesser for introducing him to Cusick's sum-of-digits conjecture on one of his first days as a PhD student in 2011.
Moreover, we wish to thank Jordan Emme and Thomas Stoll for fruitful discussions on the topic.
Finally, we thank Thomas Cusick for constant encouragement and interest in our work.
%{{{ Bibliography

%\bibliographystyle{siam}
%\bibliography{leguan}
%}}} Bibliography
\end{document}